\documentclass[12pt]{article}

\usepackage{amsthm,amsmath, amssymb}
\usepackage{amsfonts}
\usepackage{latexsym}
\usepackage{xcolor}
\usepackage{enumerate}
\usepackage{graphicx}

\usepackage{thm-restate}
\usepackage{thmtools}

\usepackage{hyperref}
\usepackage[small,bf,hang]{caption} 
\theoremstyle{plain}
\newtheorem{theorem}{Theorem}
\newtheorem{lemma}[theorem]{Lemma}
\newtheorem{corollary}[theorem]{Corollary}

\theoremstyle{definition}
\newtheorem{definition}[theorem]{Definition}

\theoremstyle{remark}

\newtheorem{rem}[theorem]{Remark}

\begin{document}

\title{\bf Polyhedra with few 3-cuts are hamiltonian}

\date{}

\author{G.~Brinkmann and C.T.~Zamfirescu\thanks{Carol T. Zamfirescu is a PhD fellow at Ghent University on the BOF (Special Research Fund) scholarship 01DI1015}\\
\small Department of Applied Mathematics, Computer Science \& Statistics \\[-0.8ex]
\small  Ghent University, 9000 Ghent, Belgium \\[-0.8ex]
\small\texttt Gunnar.Brinkmann@ugent.be, czamfirescu@gmail.com
}

\maketitle

\begin{abstract}

In 1956, Tutte showed that every planar 4-connected graph is
hamiltonian. In this article, we will generalize this result
and prove that polyhedra with at most three $3$-cuts
are hamiltonian. In 2002 Jackson and Yu
have shown this result for the subclass of triangulations.
We also prove that polyhedra with at most
four $3$-cuts have a hamiltonian path. It is well known that for each $k
\ge 6$ non-hamiltonian polyhedra with $k$ $3$-cuts exist. We give
computational results on lower bounds on the order of a possible
non-hamiltonian polyhedron for the remaining open cases of polyhedra
with four or five $3$-cuts.

\end{abstract}


\section*{Introduction}

One of the classic results in graph theory is Whitney's theorem from 1931
that 4-connected triangulations of the plane are hamiltonian \cite{Wh:31}.
This result was generalized by Tutte in 1956 showing that all $4$-connected
planar graphs are hamiltonian \cite{Tutte56}. In the meantime several stronger
versions of Whitney's theorem have been proven, with one of the strongest
results a theorem by Jackson and Yu that a hamiltonian cycle
exists even if there are up to three $3$-cuts in a triangulation \cite{JY:02}.
Although also Tutte's theorem has been generalized in several ways -- e.g.
that in a $4$-connected plane graph a hamiltonian cycle through any $2$ edges exists
\cite{Sanders97} -- the theorem by Jackson and Yu was not generalized to
all $3$-connected plane graphs with at most three $3$-cuts. We will give this proof in this article.
It should be mentioned that the theorem of Jackson and Yu is not only about
the number of $3$-cuts, but also about their relative position -- encoded by
a {\em decomposition tree}. Such decomposition trees, which are unique
for triangulations, are not defined for general plane graphs, so only the part
about the number of $3$-cuts can be generalized.

\bigskip

In this article we will use the word {\em polyhedron} for $3$-connected plane graphs
and if $G=(V,E)$ for a set $V'\subseteq V$ of vertices we write $G-V'$ to denote the subgraph
induced by the vertices in $V-V'$ and for a set $E'$ of edges we write $G-E'$ to
denote the graph $(V,E-E')$.

In section~\ref{sec:basic} we give basic definitions and results that are tools used
in the proofs in section~\ref{sec:main} where the main results are given and
proven. The central theorem in that section is Theorem~\ref{thm:3cuts} stating that
every polyhedron with at most three $3$-cuts is hamiltonian.
This theorem is also used to prove that every polyhedron
with at most four $3$-cuts is traceable. In section~\ref{sec:structure} we prove that
for the cases with four or five $3$-cuts, where it is not yet known whether
non-hamiltonian polyhedra exist, possible non-hamiltonian examples are $1$-tough,
so {\em trivial} examples as for six or more $3$-cuts do not exist. We also prove the corresponding
results for non-traceable polyhedra. Finally, in section~\ref{sec:computational} we give the
results of a computer search for non-hamiltonian polyhedra with four or five $3$-cuts
and non-traceable polyhedra with at most seven $3$-cuts.

\section{Basic definitions and results}\label{sec:basic}

\begin{definition}

\begin{itemize}

\item A vertex cutset in a graph $G=(V,E)$ is a set $S\subset V$ so that $G-S$ has more components
than $G$.\\
A vertex cutset of size $k$ in a polyhedron $G$ is called {\em trivial}, if it splits the
graph into two components, one of which is a single vertex of degree $k$ in $G$.

\item We call paths between two vertices $v,w$ {\em vertex disjoint} if they only share
the vertices $v$ and $w$. If $X$ is a set of at least two vertices, we call paths starting at $v$ and ending in
a vertex of $X$ {\em vertex disjoint} if they only share $v$.

\item A polyhedron $G$ is called {\em essentially 4-connected } if
all $3$-cuts of $G$ are trivial. This implies that between two vertices of degree larger
than $3$ there are always at least $4$ vertex disjoint paths.

\end{itemize}

\end{definition}

\begin{rem}\label{rem:notwo3intriangle}

An essentially 4-connected  polyhedron $G$ with more than $6$ vertices
does not contain
two vertices of degree $3$ that share an edge of a triangle.

\end{rem}

The (short) proof is left to the reader.

%
%

\begin{definition}

Let $G=(V,E)$ be a polyhedron and $\{u,v,w\}$ a $3$-cut in $G$.

\begin{itemize}

\item If $(V',E')$ is a component of $G-\{u,v,w\}$, then the subgraph of
$G$ induced by $V'\cup \{u,v,w\}$ is called a {\em closed component} of
$G-\{u,v,w\}$.

\item If $(V',E')$ is a closed component of $G-\{u,v,w\}$,
then\\ $(V',E'\cup\{\{u,v\},\{v,w\},\{w,u\}\})$ is called an {\em edge closed
component} of $G-\{u,v,w\}$.

\end{itemize}

\end{definition}

Note that for a $3$-cut $\{u,v,w\}$ in a polyhedron $G$ always exactly two
closed components of $G-\{u,v,w\}$ exist. A third component would imply a plane embedding of
a subdivision of $K_{3,3}$. The three vertices $u,v,w$ can not be contained in the boundary
of the same face of $G$, as otherwise a vertex in the interior of the face can be connected to all
three vertices -- allowing a plane embedding of $K_{3,3}$.

\begin{lemma}\label{lem:edgeclosed}

Let $G=(V,E)$ be a polyhedron with $k$ $3$-cuts and $\{u,v,w\}$ a $3$-cut in $G$.
Let $G'=(V',E')$ be an edge closed
component of $G-\{u,v,w\}$.
Then we have:.

\begin{itemize}

\item $G'$ is planar and the vertices $u,v,w$ form a facial triangle in the (unique) embedding
of $G'$.

\item Edge closed
components of $G-\{u,v,w\}$ are polyhedra.

\item For any two vertices in $G'$ that are not both
in $\{u,v,w\}$, there are
at least as many vertex disjoint paths in $G'$ as there are in $G$.

\item Each $3$-cut in $G'$ is also a $3$-cut in $G$.

\item All edge closed
components of $G-\{u,v,w\}$ together have at most $(k-1)$ $3$-cuts.

\end{itemize}

\end{lemma}

\begin{proof}

For $y \in V- V'$ there are $3$ vertex disjoint paths from $y$ to $\{u,v,w\}$
and as $\{u,v,w\}$ is a cut, none of these paths contains an element of
$V'-\{u,v,w\}$. The union of these paths is connected, so removing
$V- V'$ from $G$, $u,v,w$ are in the same face, so the edges
$\{u,v\},\{v,w\},\{w,u\}$ can be added (if not yet present)
inside the face where they form a facial triangle. So $G'$ is planar.
The fact that this embedding is unique will follow from the fact that $G'$
is $3$-connected.

We will now prove that for any two vertices $a,b\in V'$, $a\not\in \{u,v,w\}$
and a path $P=v_1, \dots ,v_n$ from $a$ to $b$ in $G$,
there is a path $T(P)$ from $a$ to $b$ in $G'$ using only vertices from the set
$\{v_1, \dots ,v_n\}$. We will also show that $T(P)$ uses at least $3$ vertices unless
$a=v_1$ and $v_2=b$. This implies that for a second, vertex disjoint, path $P'$ from
$a$ to $b$ we have that $P\not= P'$ implies $T(P)\not= T(P')$.
This means that for any set $\{P_1,\dots ,P_k\}$ of vertex disjoint paths
from $a$ to $b$ in $G$ the set $\{T(P_1),\dots ,T(P_k)\}$ has equal size
and is also vertex disjoint.
Assume that the path $P$ from $a$ to $b$ in $G$ is given by
$a=v_1, \dots ,v_n=b$.
Unless $v_2=b$, $v_2$ is also a vertex from
$V'$ different from $a,b$.
If $P$ only contains vertices from
$G'$, we have found our path.  Otherwise it contains a first edge
$\{v_i,v_{i+1}\}$ not in $G'$. W.l.o.g. we have $v_i=u$ and $v_{i+1}\not\in V'$.
Then there is some maximal $j>i$ with $v_j\not\in V'$ and $v_{j+1}\in \{v,w\}$.
Replacing $v_i,v_{i+1},\dots ,v_j,v_{j+1}$ by $v_i,v_{j+1}$ we have a path in
$G'$ that uses only vertices from $\{v_1, \dots ,v_n\}$. So for
$\{a,b\}\not\subset \{u,v,w\}$ there are at least $3$
vertex disjoint paths from $a$ to $b$ as $G$ is a polyhedron.

For $\{a,b\}\subset \{u,v,w\}$ -- w.l.o.g. $\{u,v,w\}=\{a,b,w\}$
-- $a,b$ and $a,w,b$ are two disjoint paths from $a$ to $b$ along the triangle.
As for each vertex $z\not\in \{u,v,w\}$ there are three
vertex disjoint paths from $z$ to $\{a,b,w\}$, there is also at least
one path from $a$ to $b$ not using $w$. So there are at least $3$
vertex disjoint paths from $a$ to $b$.

These arguments imply that $G'$ is $3$-connected, so the embedding is unique
and $\{u,v\}$, $\{v,w\}$, $\{w,u\}$ form a facial triangle in each embedding.

If $\{u',v',w'\}$ is a $3$-cut in $G'$ and $a,b$ are vertices in different components
of $G'-\{u',v',w'\}$, then at least one of them is not contained in $\{u,v,w\}$
and there is also no path from $a$ to $b$ in
$G-\{u',v',w'\}$, as there would be a corresponding path on the subset
of vertices from $a$ to $b$ in $G'-\{u',v',w'\}$.
This implies that $\{u',v',w'\}$ is also a cut in $G$. As
each $3$-cut in an edge closed component is also a $3$-cut in $G$, but $\{u,v,w\}$
is a $3$-cut in $G$ but not in any of the edge closed components,
this proves that there are at most $(k-1)$ $3$-cuts
in all edge closed components together.

\end{proof}

The number of all $3$-cuts in the edge
closed components together can be smaller than $k-1$. An easy example
is a polyhedron with a vertex $v$ of degree $3$ adjacent to
3 other vertices of degree $3$ where the neighbourhoods of these $4$ vertices
are the only $3$-cuts. Choosing the  neighbourhood of $v$ as a cut, we get
one copy of $K_4$ and one 4-connected edge closed component,
showing that the number of $3$-cuts in the edge
closed components can be smaller than $k-1$. On the other hand choosing
another cut, the edge closed component has two $3$-cuts left, showing  that there is no
unique decomposition tree like for triangulations (see \cite{JY:02}).

While in triangulations $3$-cuts are separating triangles
that lie properly inside each other, the relative position
in general polyhedra can be more complicated. It is e.g. possible that vertices
of a $3$-cut $\{u,v,w\}$ end up in different edge closed components of a $3$-cut
$\{u',v',w'\}$. This makes it worthwhile to explicitly state and prove the following
lemma, which is trivial for triangulations:

\begin{lemma}\label{lem:existstrivecc}

Each polyhedron $G=(V,E)$ with $3$-cuts has a $3$-cut $\{u,v,w\}$, so that at least
one edge closed component $G'$ of $G-\{u,v,w\}$ has no $3$-cuts -- that is: $G'$
is $4$-connected
or isomorphic to $K_4$.

\end{lemma}

\begin{proof}

Let $\{u,v,w\}$ be a $3$-cut, so that one of the edge closed components of $G-\{u,v,w\}$
has minimal cardinality among all  edge closed components of $3$-cuts of $G$.
Let $G'=(V',E')$ be this edge closed component and assume that $G'$ contains
a $3$-cut $\{u',v',w'\}$. At least one of the vertices $\{u,v,w\}$ -- w.l.o.g. $w$ --
is not in $\{u',v',w'\}$. Let $G''=(V'',E'')$ be the edge closed component
of $G'-\{u',v',w'\}$ not containing $w$. We will prove that $G''$ is also
an edge closed component of $G-\{u',v',w'\}$ -- contradicting
the minimality of $G'$.

Lemma~\ref{lem:edgeclosed} implies that $\{u',v',w'\}$ is also a $3$-cut in $G$.
As edges between vertices of $\{u,v,w\}$ are only in $G''$ if they are between vertices
of $\{u',v',w'\}$, all edges in $G''-\{u',v',w'\}$ are also in $G-\{u',v',w'\}$.
This implies that there is an edge closed component of $G-\{u',v',w'\}$
containing $G''$ and due to the minimality of $G'$ this component must
properly contain $G''$. So there is a path $v_1, \dots ,v_k$ in $G-\{u',v',w'\}$
from a vertex
$v_1$ in $G''-\{u',v',w'\}$ to a vertex $v_k$ not in $G''-\{u',v',w'\}$, so that
$v_1, \dots ,v_{k-1}$ are all in $G''-\{u',v',w'\}$. Note that $\{u',v',w'\}$ is not
a cutset in any of its edge closed components. As the edge $\{v_{k-1},v_k\}$
is not in $G'$, $v_k$ is not in $G'$, so it must be in the other edge closed
component of $G-\{u,v,w\}$, which is impossible as $v_k\not\in \{u,v,w\}$.
So we have a contradiction and $G'$ does not contain a $3$-cut.

\end{proof}

\begin{corollary}\label{cor:faceroundedge}

If $G=(V,E)$ is a polyhedron and $\{u,v,w\}$ is a $3$-cut in $G$, then
in $G-\{\{u,v\}\}$ (which is equal to $G$ if $\{u,v\}\not\in E$)
the vertices $u$ and $v$ share a face
$u,u_1,\dots ,u_k,v,v_1,\dots ,v_m$ with $k\ge 1, m\ge 1$ and
for $1\le i\le k$, $1\le j\le m$ the vertices $u_i$ and $v_j$ belong to
different components of $G-\{u,v,w\}$.

\end{corollary}

\begin{proof}
Lemma~\ref{lem:edgeclosed} implies that in the two edge closed components
the vertices $u,v,w$ form a facial triangle. Identifying the two edge closed
components along the edges of this triangle gives an embedding of a graph
$G'=(V,E')=(V,E\cup \{\{u,v\}\{v,w\}\{w,u\}\})$ containing $G$ and with the vertices
of different components of $G-\{u,v,w\}$ on different sides of the triangle
$(u,v,w)$. If one face containing the edge $\{u,v\}$ is $v,u,u_1,\dots ,u_k$ (so $k\ge 1$
as the smallest possible face size is a triangle) and the other
is $u,v,v_1,\dots ,v_m$ (so $m\ge 1$), in $G''=(V,E'-\{u,v\})$
we get the face as described in the
statement. As $(V,E\cup \{u,v\})$ is $3$-connected, the embedding of
$(V,E\cup \{u,v\})$ is unique, so equivalent to the subembedding of $G'$ and
$u,u_1,\dots ,u_k,v,v_1,\dots ,v_m$ is also a face in $G-\{\{u,v\}\}$.
As $G$ is a subgraph of $G'$, for $1\le i\le k$, $1\le j\le m$ the vertices $u_i$ and $v_j$ belong to
different components of $G-\{u,v,w\}$.

\end{proof}

\begin{lemma}\label{lem:remove_vertices}

A polyhedron $G$ with $k$ $3$-cuts contains a spanning subgraph that can be obtained from a
$4$-connected polyhedron by deleting at most $k$ vertices.

\end{lemma}

\begin{proof}

We will prove this by induction in $k$. For $k=0$ the statement is trivial, so assume that $k>0$
and that $\{u,v,w\}$ is a $3$-cut in $G$.
If $u,v,w$ do not form a separating triangle, assume w.l.o.g. that  there is no edge $\{u,v\}$.
We will construct a graph $G'$ with at most
$k-1$ $3$-cuts in which a vertex can be removed to obtain $G-\{\{u,v\}\}$. By induction a spanning subgraph
of $G'$ can be obtained by removing at most $k-1$ vertices from a $4$-connected polyhedron,
proving the result.

With the notation of Corollary~\ref{cor:faceroundedge} for the face of $G-\{\{u,v\}\}$
containing $u$ and $v$ we add a new vertex $z$ in the interior of
the face $u,u_1,\dots ,u_k,v,v_1,\dots ,v_m$ and connect it to all vertices in the boundary
to obtain a graph $G'$. As $z$ connects vertices from the two different components of
$G-\{u,v,w\}$, the set $\{u,v,w\}$ is not a cutset any more.

Assume now that there is a $3$-cut $\{u',v',w'\}$ in $G'$ that is not already a cutset in $G$.
As the degree of $z$ is at least $4$, $z$ cannot form a trivial component and
each component contains at least one vertex from $G$.
Let $w_1\not= z, w_2\not= z$ be vertices from different components of $G'-\{u',v',w'\}$. So there is a path
from $w_1$ to $w_2$ in $G-\{u',v',w'\}$ which is not in $G'-\{u',v',w'\}$. As $\{u,v\}$ is the only
edge that might be in $G-\{u',v',w'\}$ but not in  $G'-\{u',v',w'\}$ this edge $\{u,v\}$
is contained in this path and therefore $\{u,v\}\cap \{u',v',w'\}=\emptyset $.
As  $\{u,v\}\in E(G)$ due to the choice of $u,v$ the vertices $u,v,w$ form a separating triangle
in $G$.
We also have $z\in \{u',v',w'\}$ as otherwise the edge $\{u,v\}$ could be replaced by the path $u,z,v$.
For the same reason $u$ and $v$ belong to different components in $G'-\{u',v',w'\}$.
This again implies that $w\in \{u',v',w'\}$, so w.l.o.g. $\{u',v',w'\}=\{u',z,w\}$ for some vertex $u'$.

As $G-\{u,v,w\}$ has exactly two components, the interior and exterior of the separating
triangle are connected and w.l.o.g. $u'$ is in the exterior. So the interior is also
connected in $G'-\{u',z,w\}$. As $u,v$ both have neighbours in the interior, they belong
to the same component -- a contradiction.

So there is no $3$-cut $\{u',v',w'\}$ in $G'$ that is not already a cutset in $G$ and $G'$
has at most $(k-1)$ $3$-cuts.

\end{proof}

A graph $G=(V,E)$ is {\em $k$-hamiltonian} if for each set $S\subset V$ with $|S|\le k$
the graph $G-S$ is hamiltonian.
In \cite{ThomasYu} Thomas and Yu prove the following result which was
originally conjectured by Plummer:

\begin{theorem}\label{thm:RYU}

$4$-connected planar graphs are $2$-hamiltonian.

\end{theorem}

Together with Lemma~\ref{lem:remove_vertices} this theorem implies immediately:

\begin{theorem}

\begin{itemize}

\item Polyhedra with at most two $3$-cuts are hamiltonian.

\item Polyhedra with at most one $3$-cut are $1$-hamiltonian.

\end{itemize}

\end{theorem}

Theorem~\ref{thm:RYU} can obviously not be strengthened to imply
$3$-hamiltonicity, so in order to prove that polyhedra with at most three $3$-cuts are hamiltonian,
we need another strategy.

\medskip

We use the following technical Lemma~\ref{lem:JY:02-6} from
\cite{JY:02}. A {\em circuit graph} is a pair $(G,F)$
such that $G$ is a 2-connected plane graph and $F$ is a facial cycle
of $G$ such that for any 2-cut $U$ of $G$, each component of $G - U$
contains a vertex of $F$. In a $3$-connected graph, each facial cycle has this property.
Let $X$ be a cycle in a circuit graph
$(G,F)$. An $X$-bridge of $G$ is either a single edge of $G-E(X)$ with
both ends on $X$, or a component $B$ of $G-V(X)$ together with the edges
with one endpoint on $X$ and one in $B$ and the endpoints of these edges
on $X$. We say that a cycle $X$ is an
$F$-Tutte cycle if for any $X$-bridge $B$ we have that $|V(B) \cap
V(X)|\leq3$ and for any $X$-bridge $B$ containing an edge of $F$ we
have that $|V(B) \cap V(X)|\leq2$. Using this terminology, an
$F$-Tutte cycle is a hamiltonian cycle if and only if all bridges are
single edges. We call such bridges trivial.

We will use the following Remark:

\begin{rem}\label{rem:unionpaths}

If $(G,F)$ is a circuit graph and $X$ is an $F$-Tutte cycle, then we have:

\begin{itemize}

\item If $v\in F$ and $v\not\in X$ then there are at most
$2$ vertex disjoint paths from $v$ to $X$. So there
is a set $C$ of at most two vertices, so that $v$ and the elements
of $X-C$ are in different components of $G-C$.

\item If $v\not\in F$ and $v\not\in X$ then there are at most
$3$ vertex disjoint paths from $v$ to $X$. So there
is a set $C$ of at most three vertices, so that $v$ and the elements
of $X-C$ are in different components of $G-C$.

\end{itemize}

\end{rem}

\begin{proof}

This is a direct consequence of the definition of an $X$-bridge.


\end{proof}

\section{The main results about hamiltonian cycles and paths in polyhedra with few 3-cuts}\label{sec:main}

In this section we will first prove some technical results. They will lead to the main theorem that polyhedra with
at most three $3$-cuts are hamiltonian, but can also be useful in other contexts.

\begin{lemma}[\cite{JY:02}, (3.2)]\label{lem:JY:02-6}
Let $(G,F)$ be a circuit graph, $r, z$ be vertices of $G$ and $e\in
E(F)$. Then $G$ contains an $F$-Tutte cycle $X$ through $e$, $r$ and
$z$.
\end{lemma}

A key to the main result in \cite{JY:02} is their Theorem~4.1.
The following theorem is a corresponding result in the
more general context of polyhedra. For $t=d=0$ it is also a direct consequence
of Corollary~2 in \cite{Sanders97}.

\begin{theorem} \label{thm:2edgesplus}

Let $0\le t,d \le 2$ be integers with $t+d\le 2$ and
$G=(V,E)$ be 
an essentially $4$-connected  polyhedron.

Let $f$ be a face of $G$ containing a vertex $v$ with degree
at least $4$
with neighbouring edges $\{u,v\}$, $\{v,w\}$ in the boundary
and assume that except for $u,w$ -- which can have any degree
of at least $3$ --
there are exactly $t$ vertices $r_1,\dots ,r_t$ of degree $3$.
Furthermore let $t_1=(u_1,v_1,w_1), \dots ,t_d=(u_d,v_d,w_d)$ be
facial triangles different from $f$, so that in case $d=2$
not all three faces $f,t_1,t_2$ share a vertex of degree
$3$ (which would have to be $u$ or $w$).

Then the following two equivalent statements are true:

\begin{itemize}

\item There is a  hamiltonian cycle of $G$ containing $\{u,v\}$ and $\{v,w\}$
and $d$ edges $e_1, \dots ,e_d$ with $e_j\in t_j$ for $1\le j\le d$ so that all edges
$\{u,v\}, \{v,w\},e_1, \dots ,e_d$ are different.

\item There is a path from $u$ to $w$ containing all vertices of $G$ but
$v$, that contains  $d$ distinct edges
$e_1, \dots ,e_d$
with $e_j\in t_j$ for $1\le j\le d$.
Note that if $f$ is a triangle, the path contains no edge of $f$.

\end{itemize}

\end{theorem}

\begin{proof}


As $G$ is $3$-connected, the edge $\{u,w\}$ is either not present in $G$
or the face $f$ is a triangle containing the edge $\{u,w\}$. Otherwise
$\{u,w\}$ would be a 2-cut. So we can assume that $f$
is a triangle $(u,v,w)$. If this was not the case, we could add the edge $\{u,w\}$
and remove it afterward as it will not be part of the hamiltonian cycle constructed.

Furthermore we can assume that $t_1, \dots ,t_d$ do not contain vertices of degree $3$:
If a triangle contains a vertex of degree $3$, one of its edges will be contained in a
hamiltonian cycle, so that we do not have to ensure this property.
Note that this edge is only counted for one triangle, as for
$d=2$, there are no vertices of degree $3$ (except maybe $u$ or $w$
which must not be shared). In case one
triangle contains $u$ and $u$ has degree $3$ (or analogously $w$),
we have also to make sure that it
contains another edge than $\{v,u\}$, but this is guaranteed as the edge $\{u,w\}$ is not contained
in the hamiltonian cycle and the third edge is contained in both faces different from $f$.

It is sufficient to prove the theorem for $d=0$, as for $d>0$
we can add a new vertex of degree $3$ in the interior
of triangle $t_i$,  $1\le i \le d$. As no vertex in the triangle has degree
$3$, the result is an essentially $4$-connected  polyhedron with $d=0$
and $t$ increased by $d$. At the end we can modify the hamiltonian
cycle through the new vertices by replacing the two edges incident with
the new vertices by the edge connecting the endpoints of the edges
on the triangle. The resulting hamiltonian cycle has the properties given in the
theorem.

If $t=0$ we choose arbitrary vertices $r_1,r_2\not\in \{u,v,w\}$ and if $t=1$ we choose
an arbitrary vertex $r_2\not\in \{u,v,w,r_1\}$.

Let $G'$ denote the graph $G-\{v\}$.
If $v$ was not contained in a $3$-cut, $G'$ is $3$-connected
and with $F'$ the new face, $(G',F')$ is an $F'$-circuit graph.

If $v$ was contained in a $3$-cut, the 2-cuts of $G-\{v\}$ are exactly
the neighbourhoods of former vertices of degree $3$ that were adjacent
to $v$.
Together with $v$ each 2-cut is a $3$-cut in $G$, so we have exactly
2 components, one component
is trivial, and the cutvertices are on the boundary of $F'$.
The face $F'$ contains at least $4$ vertices and after the removal of a 2-cut
each component contains vertices of $F'$.
So $(G',F')$ is an $F'$-circuit graph.

Due to Lemma~\ref{lem:JY:02-6} $G'$ has an $F'$-Tutte cycle
$X$ through $\{u,w\}$, $r_1$ and $r_2$. We will show that $X$
is a hamiltonian cycle.



Assume first that there is a vertex $y\in F'$, $y\not\in X$.
As $X$ is an $F'$-Tutte cycle, this would imply that
there is a 2-cut $C$ separating $y$ from
the vertices in $\{u,w,r_1,r_2\}-C$. As there are at least two vertices in
$\{u,w,r_1,r_2\}-C$ and as one component of each 2-cut is trivial
-- which can not be $y$ -- at least one vertex of $\{u,w,r_1,r_2\}-C$ is in
the same component as $y$. This is a contradiction, so all vertices of $F'$
are contained in $X$.

Assume now that $y$ is not on $F'$ and not on $X$.
This implies that $\deg(y)\ge 4$ and that there are $4$ vertex disjoint paths
in $G$ from $y$ to $v$. All these paths must contain vertices of $F'\subseteq X$, so
the union of the first parts of these paths from $y$ to the first element from $X$
shows that the $X$-bridge containing $y$ has at least $4$ endpoints on $X$
-- a contradiction.

So $X$ is a hamiltonian cycle of $G'$ and replacing the edge $\{u,w\}$ by the
edges $\{u,v\}$ and $\{v,w\}$, we have a hamiltonian cycle
of $G$ with the required properties.

\end{proof}

Theorem~\ref{thm:2edgesplus} can be used to prove the following corollary:

\begin{corollary}\label{cor:oneedgetriangle_deg3}

Let $G=(V,E)$ be an essentially $4$-connected  polyhedron with at most one
vertex of degree $3$ and containing a facial triangle $(u,v,w)$
with degree of $v$ at least $4$.
Then the following two equivalent statements are true:

\begin{itemize}

\item There is a  hamiltonian cycle of $G$ containing $\{u,w\}$ as only edge of the
triangle $(u,v,w)$. If there is no vertex with degree $3$ and
another triangle $(\bar u,\bar v,\bar w)$ is given, the cycle can be chosen in a way
that also a (different) edge of $(\bar u,\bar v,\bar w)$ is contained.

\item There is a hamiltonian path from $u$ to $w$ containing no edge of the triangle
$(u,v,w)$. If there is no vertex with degree $3$ and
another triangle $(\bar u,\bar v,\bar w)$ is given, the path can be chosen in a way
that also an edge of $(\bar u,\bar v,\bar w)$ is contained.

\end{itemize}

\end{corollary}

\begin{proof}

As $v$ has degree at least $4$ there are
$\{u',v\},\{v,w'\}$ that are both not contained in the triangle $(u,v,w)$, but another face $f$.
Applying Theorem~\ref{thm:2edgesplus} for $d=1$ to the edges $\{u',v\},\{v,w'\}$
and the triangle $t_1=(u,v,w)$, (resp. $d=2$ if $(\bar u,\bar v,\bar w)$ is given)
we get that there is a hamiltonian cycle through
$\{u',v\},\{v,w'\}$ that contains an edge of $(u,v,w)$ -- but this edge can only be
$\{u,w\}$. In case $(\bar u,\bar v,\bar w)$ is given, the cycle also contains a
(different) edge of $(\bar u,\bar v,\bar w)$.
\end{proof}

\begin{lemma}\label{lem:onecut}

Let $G$ be a polyhedron with at most one $3$-cut and let $(u,v,w)$
be a triangular face with $\deg(v)\ge 4$.

\begin{description}
\item [(i)] Then there is a hamiltonian cycle of $G$ containing $\{u,v\}$ and $\{v,w\}$.
This implies that there is a path from $u$ to $w$ containing no edges of the triangle
$(u,v,w)$, but all vertices except $v$.
If another triangle $(\bar u,\bar v,\bar w)$ is given, the cycle (resp. path) can be chosen in a way
that also a (different) edge of $(\bar u,\bar v,\bar w)$ is contained.

\item[(ii)] There is also a hamiltonian
cycle of $G$ containing $\{u,w\}$ as only edge of $(u,v,w)$.
So there is a hamiltonian path from $u$ to $w$ containing no edges of $(u,v,w)$.

\end{description}

\end{lemma}

\begin{proof}
We will give the proof of (i) only for the stronger case where also $(\bar u,\bar v,\bar w)$
is given.

Note that as $(u,v,w)$ and $(\bar u,\bar v,\bar w)$ are triangular faces,
none of them forms the unique $3$-cut
and they are also faces in one of the edge closed components.

If the $3$-cut is trivial, (i) and (ii) are direct consequences of Theorem~\ref{thm:2edgesplus}
and Corollary~\ref{cor:oneedgetriangle_deg3}.

So assume that the $3$-cut $\{u',v',w'\}$ is not trivial and let $G'$, $G''$ be the two edge-closed
components of $G-\{u',v',w'\}$ sharing $\{u',v',w'\}$.
Then $G'$, $G''$ are $4$-connected and all vertices have degree at least $4$.
Assume that -- w.l.o.g. -- $u,v,w$ are in $G'$.

{\bf (i) } $G'$
contains a hamiltonian cycle $H$ through $\{u,v\}$ and $\{v,w\}$ that contains also
an edge of the triangle $(u',v',w')$ different from $\{u,v\}$ and $\{v,w\}$
(Theorem~\ref{thm:2edgesplus}) and a different edge $\bar e$ in $(\bar u,\bar v,\bar w)$
if that triangle is in $G'$.
If it contains one edge -- say $\{u',v'\}$ -- different from $\{u,v\}$, $\{v,w\}$, and $\bar e$,
then we can apply Theorem~\ref{thm:2edgesplus} to $G''$ and get that $G''$
has a path from $u'$ to $v'$ containing all vertices
but $w'$. Replacing the edge $\{u',v'\}$ with this path gives the hamiltonian cycle of $G$
with the required properties. If $(\bar u,\bar v,\bar w)$ was in $G''$, the path
can be chosen in a way that it contains an edge of  $(\bar u,\bar v,\bar w)$ not contained in $(u',v',w')$.

If $H$ contains two edges of $(u',v',w')$
different from $\{u,v\}$ and $\{v,w\}$ -- say $\{u',v'\}$, $\{v',w'\}$ --
then we can apply Corollary~\ref{cor:oneedgetriangle_deg3} to conclude that $G''$
has a hamiltonian path from $u'$ to $w'$ using no edges of the triangle $(u',v',w')$
(but an edge of $(\bar u,\bar v,\bar w)$ if that was in $G''$).
Replacing the path $u',v',w'$ in $H$ by this hamiltonian path gives the hamiltonian cycle of $G$
with the required properties.

{\bf (ii) }
Let $\{v,x\},\{v,y\}$ be edges in the same face of $G'$ with $\{x,y\}\cap \{u,w\}=\emptyset$.
Then there is a hamiltonian cycle $H$ through $\{v,x\}$ and $\{v,y\}$ that contains also
an edge of the triangle $(u,v,w)$ (which must be $\{u,w\}$)
and in case $\{u',v',w'\}\not= \{v,u,y\}$ also a different
edge in $(u',v',w')$ (Theorem~\ref{thm:2edgesplus}). Precisely as in the first part, this
hamiltonian cycle can be extended to a hamiltonian cycle of $G$ by applying Theorem~\ref{thm:2edgesplus},
resp. Corollary~\ref{cor:oneedgetriangle_deg3}. The extension in $G''$ does not use any edges
that occur in $G'$ (also no edges of the triangle $(u',v',w')$ which occur in both -- $G'$ and
$G''$), so the resulting hamiltonian cycle contains $\{u,w\}$ as only edge from $(u,v,w)$.

\end{proof}

\begin{lemma}\label{lem:twocuts}

Let $G$ be a polyhedron with at most two $3$-cuts and let $(u,v,w)$
be a triangular face.

Then there is a hamiltonian cycle of $G$ containing
an edge of $(u,v,w)$.

\end{lemma}

\begin{proof}

If all $3$-cuts are trivial, this is a direct consequence of Theorem~\ref{thm:2edgesplus},
so suppose that $\{u',v',w'\}$ is a nontrivial $3$-cut, that $G'$, $G''$ are the two edge
closed components and that $G'$ contains the face $(u,v,w)$.

As $G'$ and $G''$ together have at most one $3$-cut, at most one of $u',v',w'$
has degree $3$ in one of them,
so assume that $\deg_{G'}(v')\ge 4$ and $\deg_{G''}(v')\ge 4$.
Lemma~\ref{lem:onecut} (i) implies that there is a path from $u'$ to $w'$ in $G'$
containing all vertices but $v'$ that contains also an edge of $(u,v,w)$.
In $G''$ we can apply Lemma~\ref{lem:onecut} (ii)
to conclude that $G''$ contains a hamiltonian path from $u'$ to $w'$ using
only edges of $G''$ not contained in the triangle $(u',v',w')$.

The union of these paths gives the hamiltonian cycle with the described properties.

\end{proof}

\begin{theorem}\label{thm:3cuts}

Every polyhedron $G=(V,E)$ with at most three $3$-cuts is hamiltonian.

\end{theorem}

\begin{proof}

If all $3$-cuts are trivial, this is a direct consequence of Theorem~\ref{thm:2edgesplus}
with $t=2$ and choosing one of the vertices with degree $3$ as $u$.

So assume that there is a nontrivial cut $\{u,v,w\}$. Splitting $G$ at this cut we get
two edge closed components $G'$ and $G''$ with together at most two
$3$-cuts.

If one of them -- w.l.o.g. $G''$ -- has no $3$-cuts, then $G'$ has at most
two $3$-cuts and we can apply Lemma~\ref{lem:twocuts}.
We conclude that $G'$ has a hamiltonian cycle containing an edge of $(u,v,w)$.
Assume that $\{u,w\}$ is contained.
As $G''$ does not have $3$-cuts, we can
replace $\{u,w\}$ by a path from $u$ to $w$ in $G''$ containing all vertices but $v$
(Theorem~\ref{thm:2edgesplus}).

So assume that $G'$ and $G''$ both have one $3$-cut. This implies that
in each of $G', G''$ at most one
of $u,v,w$ can have degree $3$. So one vertex -- w.l.o.g. $v$ has degree $4$
in both edge-closed components. Then Lemma~\ref{lem:onecut} implies that
$G'$ has a path
from $u$ to $w$ containing
no edges of $(u,v,w)$ and all vertices but $v$ and that
$G''$ has a hamiltonian path from
$u$ to $w$ containing no edges of $(u,v,w)$. Combining these paths gives
a hamiltonian cycle of $G$.

\end{proof}

\begin{corollary}

If all polyhedra with at most $k$ $3$-cuts are hamiltonian, then all polyhedra with at
most $(k+1)$ $3$-cuts are traceable (that is: contain a hamiltonian path).

Especially: all polyhedra with at most four $3$-cuts are traceable.

\end{corollary}

\begin{proof}

Let $G=(V,E)$ be a polyhedron with at most $(k+1)$ $3$-cuts.

Theorem~\ref{thm:2edgesplus} implies
that for any
two vertices $u,v$ in a triangle $(u,v,w)$ of a $4$-connected polyhedron,
there is a path from $u$ to $v$ through all vertices but $w$.
The same is obviously also true for $K_4$.

Let $\{u,v,w\}$ be a $3$-cut in $G$ so that one of the edge
closed components, say $G'$, has no $3$-cut. The other
edge closed component $G''$ has at most $k$ $3$-cuts,
so there is a hamiltonian cycle $H$ in $G''$. There are 3 possibilities
how this hamiltonian cycle passes through edges of the triangle
$(u,v,w)$:

\begin{itemize}

\item $H$ uses no edge of $u,v,w$:\\
So w.l.o.g. $H=u,u_1, \dots ,v,v_1, \dots ,w,w_1,\dots ,u$.
Combining $H-\{\{u,u_1\},\{v,v_1\}\}$ with a path
in $G'$ from $u$ to $v$ through all vertices but $w$ gives a
hamiltonian path in $G$ from $u_1$ to $v_1$.

\item $H$ uses only the edge $\{u,v\}$ of $u,v,w$:\\
Combining $H-\{\{u,v\}\}$ with a path
in $G'$ from $u$ to $v$ through all vertices but $w$
even gives a hamiltonian cycle in $G$.

\item $H$ uses the edges $\{u,v\},\{v,w\}$ of $u,v,w$:\\
Combining $H-\{\{u,v\},\{v,w\}\}$ with a path
in $G'$ from $u$ to $v$ through all vertices but $w$
gives a
hamiltonian path in $G$ from $v$ to $w$.

\end{itemize}

\end{proof}

On the other hand one can prove:

\begin{lemma}

For each $d\ge 8$ there exist non-traceable triangulations with $d$ $3$-cuts.

\end{lemma}

\begin{proof}
For each $n\ge 6$ there is a double wheel that provides an example of a
$4$-connected triangulation with $n$ vertices and $2n-4$ triangular faces.


For even $d\ge 8$ take a double wheel (or any other $4$-connected triangulation)
$T$ with $d$
triangular faces. $T$ has
$n= \frac{d}{2} +2$ vertices. As $d\ge 8$ we have $d - n \ge 2$. If we insert single vertices
into $d$ triangular faces of $T$ and connect them to the vertices in the
surrounding triangle, we obtain a triangulation $T'$, that has exactly $d$
$3$-cuts: each triangle into which a vertex was inserted is a $3$-cut
and for each other set $C$ of $3$ vertices, the subgraph $T$ is still connected and
each new vertex that is not in $C$ is connected to at least one
remaining vertex in $T-C$ -- so $T'-C$ is connected. This proves that the
$d$ triangles form the only $3$-cuts.

Removing the $n$ vertices of $T$, we have $d\ge n+2$ components left, which proves
that there can not be a hamiltonian path as it would have to pass all $d$ components
with less than $d-1$ intermediate points from $T$, which is obviously impossible.

Denoting the number of components
of a graph as  $c(G)$, Hendry \cite{scattering} defined the {\em scattering number} $s(G)$ as
$s(G)=\max\{c(G-X)-|X| \quad | \quad X\subset V(G), c(G-X)\ge 2\}$. The
scattering number can be considered as a generalization of the concept
of toughness. So in the language of Hendry \cite{scattering} we have
proven that $s(T')\ge 2$ and illustrated the well known fact that it has to
be at most $1$ to be traceable.  Graphs with scattering number at
least $2$ can be considered trivially non-traceable.

\end{proof}

\section{The toughness and scattering number of polyhedra with few 3-cuts}\label{sec:structure}

As for $6$ or more $3$-cuts there are non-hamiltonian polyhedra even in
the subclass of triangulations \cite{towardsWhitney}, the only remaining cases for which it
is not decided whether non-hamiltonian polyhedra with these numbers of
cuts exist are four or five $3$-cuts. For traceability the undecided cases
are five, six or seven $3$-cuts. In this chapter we will prove that there
are no trivially non-hamiltonian (not $1$-tough) polyhedra with four or five $3$-cuts
and no trivially non-traceable (scattering number at least two)
polyhedra with five, six or seven $3$-cuts.

In \cite{towardsWhitney} the authors
prove that triangulations with $4$ or $5$ $3$-cuts are necessarily
$1$-tough, which means that no easy counterexamples exist. This
result is based on the following lemma.

A graph $G=(V,E)$ is $1$-tough if $s(G)\le 0$, so for all cutsets $S\subset V$ of $G$ we have that $c(G-S)\le |S|$.
Note that in a triangulation a $3$-cut consists of vertices of a separating triangle and the vertices of each
separating triangle form a $3$-cut.

\begin{lemma} \label{rem:1toughk} \cite{towardsWhitney}

In a triangulation $T=(V,E)$ with at most $d$ separating triangles we have
for each splitting set $S\subset V$: $c(T-S) \le  |S|-2+\lfloor \frac{d}{2}\rfloor $.

So for $d\le 5$ we have $c(T-S) \le |S|$ which implies that triangulations with at most
$5$ separating triangles are $1$-tough.

For $d\le 7$ we have $c(T-S) \le |S|+1$ which implies that for triangulations with at most
seven $3$-cuts the scattering number is at most $1$.

\end{lemma}

One may hope that as polyhedra can have much less edges than triangulations,
the situation is different and that there might be
polyhedra with four or five $3$-cuts that are not $1$-tough or polyhedra
with five, six or seven $3$-cuts and scattering number at least $2$.
We will now prove that this is not the case.

\begin{rem}\label{remark:totriang}

For each polyhedron $G=(V,E)$ with $d$ $3$-cuts
and each splitting set $S\subset V$, there is a triangulation
$T_{S,G}=(V,E')$ containing $G$ as a subgraph with at most $d$ $3$-cuts,
so that $c(T_{S,G}-S)=c(G-S)$.

\end{rem}

\begin{proof}

First note that for two non-adjacent vertices $v,v'$ of a face $f$ of $G$,
we have that $\{v,v'\}\not\in E$, as otherwise $\{v,v'\}$ would be a 2-cut.
This implies that we can add the edge $\{v,v'\}$ inside $f$ without producing
double edges. Adding an edge, no new cuts can occur, so that the number
of $3$-cuts can not increase.

Let $A_1,\dots ,A_k$ be the components of $G-S$. We call a pair $\{v,v'\}$
of vertices $v,v'\in V$ connecting if $v$ and $v'$ are in different components
of $G-S$ and non-connecting otherwise. Obviously all edges of $G$
are non-connecting. For any graph $G'=(V,E')$ containing $G$ that has only
non-connecting edges we have that $c(G'-S)=c(G-S)$.

Let now $T=(V,E')$ be a polyhedron that has only
non-connecting edges, contains $G$ and has the maximum number of edges
among all polyhedra with these properties.

Assume that $T$ has a face $f$ that is not a triangle. Then $f$ does not contain
vertices from $S$, as a vertex $v\in S$ can be connected to any other
vertex in $f$ -- except to its two neighbours -- by a non-connecting edge.
This is in contradiction to the maximality of $T$. The same would be true
if $f$ contained only vertices from the same component of $G-S$.
So $f$ contains vertices from different components of $G-S$. As there
are no edges in $G$ or $T$ between different components of $G-S$,
they must be separated by vertices of $S$. This contradicts the fact that
no vertices of $S$ are in $f$.

So all faces are triangles and $T$ is a triangulation with the desired properties.

\end{proof}

Note that the triangulation $T_{S,G}$ is in general not uniquely determined.

\begin{corollary}

In a polyhedron $G=(V,E)$ with at most $d$ $3$-cuts we have
for each splitting set $S\subset V$: $c(G-S) \le  |S|-2+\lfloor \frac{d}{2}\rfloor $.

So for $d\le 5$ we have $c(G-S) \le |S|$ which implies that polyhedra with at most
five $3$-cuts are $1$-tough.
For $d\le 7$ we have $c(T-S) \le |S|+1$ which implies that for polyhedra with at most
seven $3$-cuts the scattering number is at most $1$.

\end{corollary}

\begin{proof}

With Remark~\ref{remark:totriang} and Lemma~\ref{rem:1toughk} this follows directly, as
for each set $S$ we have

$c(G-S) \le  c(T_{S,G}-S) \le |S|-2+\lfloor \frac{d}{2}\rfloor $.

\end{proof}

\section{Computational results}\label{sec:computational}

The case of four or five $3$-cuts is even in the subclass of triangulations, where one has
more information on the polyhedron, still undecided. One would expect that -- in case
polyhedra resp. triangulations with four or five $3$-cuts are all hamiltonian --
the proof for triangulations is easier and found first.

If on the other hand there are polyhedra with four or five $3$-cuts that are non-hamiltonian,
it is possible that they only exist outside the class of triangulations or
one might expect the size of a smallest counterexample outside the class of triangulations
to be smaller than the size of the smallest triangulation that is a counterexample, as
removing edges from triangulations (without increasing the number of $3$-cuts)
might destroy hamiltonian cycles. To this end or to get at least a lower bound
on the size of a possible non-hamiltonian polyhedron with four or five $3$-cuts, we implemented
a computer search. This computer search is based on the program {\em plantri} \cite{plantri}
to which we added a filter that allows to restrict the generation to polyhedra with a number
of $3$-cuts inside a given range (in our case $4$ or $5$). The program generates more than $500,000$
non-isomorphic polyhedra -- each with four or five $3$-cuts -- per second on an Intel Quad CPU Q8200
running at 2.33GHz.

In plantri polyhedra are generated by successively removing edges from triangulations
but keeping the graph $3$-connected. This implies that the number of $3$-cuts can never decrease
when removing edges and no further edges have to be removed in the recursive proces
when the number of $3$-cuts exceeds the upper limit. Polyhedra are output
only if the number of $3$-cuts lies within the range given by the upper and lower limit, but
polyhedra with fewer $3$-cuts than the lower limit also occur inside the generation process.

We count $3$-cuts by using the following easy lemma:

\begin{lemma}\label{lem:test}

Three vertices $u,v,w$ form a $3$-cut in a polyhedron $G=(V,E)$ if
and only if
any two of them share a face, but not all three of them share a face.

\end{lemma}

\begin{proof}

Assume first that $\{u,v,w\}$ is a $3$-cut. If all three shared a face, this would make a plane embedding of
$K_{3,3}$ possible, so they don't share a face. On the other hand the two edge closed components
can be identified along the triangles $(u,v,w)$ to give an embedding of  $(V,E\cup \{\{u,v\},\{v,w\},\{w,u\}\})$,
which is unique and in which any two of them share a face. So any two of them also share a face
in $G$, which can be obtained by possibly deleting edges.

Assume now that each two of them, but not all three share a face. Then we can connect each pair of
them (if not yet connected) by an edge inside the corresponding face and obtain a triangle
$(u,v,w)$. If this was a facial triangle, then all three of them would share a face, so it is a non-facial
triangle separating the inside from the outside and $u,v,w$ form a $3$-cut.

\end{proof}

For each triangulation we start with, we store for each vertex $v$ the set of vertices
that $v$ shares
a face with. This set is stored as a bit set. Furthermore we store each set of $3$ vertices
that share a face -- that is: all sets of facial triangles. Then the condition from Lemma~\ref{lem:test}
is tested for all triples $u<v<w$ of vertices and the $3$-cuts are counted.

For each number $i>0$ of edges removed from a triangulation the vector of bit sets coding for
each vertex the set of vertices it shares a face with is updated from the corresponding vector for
$i-1$. The only vertices for which the sets change are vertices in the new face that is the union
of the two faces $f_1,f_2$ on two sides of the edge just removed. Possible new $3$-cuts are only
$3$-cuts that contain one vertex $u$ from $f_1$ (but not $f_2$), one vertex $w$ from $f_2$ (but not $f_1$)
and one vertex neither in $f_1$, nor in $f_2$. If the vector storing the bit sets of vertices is \verb+bs[]+ --
so \verb+bs[u]+ is the set of all vertices that $u$ shares a face with -- and \verb+newface+ is a bit set
with all vertices in the new face, then for each $u,v$ as above counting the $3$-cuts they are involved in
comes down to counting the $1$-bits in \verb+bs[u] & bs[v] & ~newface+ which is a very
efficient operation. Note that all these $3$-cuts are new in the sense that they were no $3$-cuts before
the last edge was removed. If they had been $3$-cuts before, after removal the graph had at least $3$
components and would either allow a plane embedding of $K_{3,3}$ or have a $2$-cut.

When generating only polyhedra with four or five $3$-cuts, the time needed for determining
the number of $3$-cuts is about $5\%$ of the generation time. We tested the implementation
by comparing the results to those of a very simple filter that filtered all polyhedra by removing
all sets of three vertices and testing the rest for being connected. The results were compared
for all polyhedra on up to $15$ vertices (more than $25,000,000,000$ graphs) by once counting
$3$-cuts with the simple program and once generating only those with a given number of cuts.
There was complete agreement.

To check hamiltonicity, we added a filter to plantri.
The filter uses a simple branch and bound approach trying to
build a hamiltonian cycle by successively increasing a path. There are only two
look-aheads used. They are also implemented by very efficient bit operations.:

\begin{itemize}

\item If not all vertices have been visited and the start vertex of the path
has no neighbour in the set of unvisited vertices, the routine backtracks.

\item If one of the still unvisited vertices has less than two neighbours in
the set consisting of the still unvisited vertices and the start and end vertex
of the path, the routine backtracks.

\end{itemize}

The routine takes advantage of the way that plantri constructs
polyhedra: for triangulations a hamiltonian cycle is searched and
stored in case one is found.  As all other polyhedra are constructed
by removing an edge from the ancestor in the generation, a previously
constructed hamiltonian cycle may still be present.  So for polyhedra
that are not triangulations and with a hamiltonian ancestor, it is
first tested whether the last edge removed belonged to the hamiltonian
cycle stored for the ancestor of the polyhedron.  Only if this is not
the case, a new hamiltonian cycle is constructed and stored. The profile
for up to $15$ vertices showed that hamiltonicity testing took less than
$20\%$ of the total time, but the ratio increased slowly with the number of
vertices. The hamiltonicity testing routine was checked by comparing the
number of hamiltonian polyhedra on up to $15$ vertices to numbers
obtained by an extremely simple independent program. There was complete agreement.

Analogous to the case of four or five $3$-cuts where the existence of hamiltonian cycles
is still open, the existence of hamiltonian paths for five, six or seven $3$-cuts is still open.

The filter checking traceability is very similar to the one checking hamiltonicity, except that it starts at every
vertex and uses slightly modified bounding criteria.
With the set $S$ consisting of the still unvisited vertices and the end vertex of the path
(but not the start vertex), the only bounding criteria used are:

\begin{itemize}

\item If one of the still unvisited vertices has no neighbour in $S$
the routine backtracks.

\item If at least two of the still unvisited vertices have less than two neighbours in $S$, then
the routine backtracks.

\end{itemize}

The profile
for up to $14$ vertices showed that traceability testing took about
$22\%$ of the total time, but also here the ratio increased slowly with the number of
vertices. The traceability testing routine was checked by comparing the
number of non-traceable polyhedra on up to $15$ vertices to numbers
obtained by an extremely simple independent program. There was complete agreement.
The first
non-traceable polyhedra occur for $14$ vertices.

The result of the computations performed on a cluster of the HPC infrastructure at Ghent University
are:

\begin{lemma}

\begin{itemize}
\item There are no non-hamiltonian polyhedra with at most five $3$-cuts on up to $19$ vertices.

\item There are no non-traceable polyhedra with at most seven $3$-cuts on up to $18$ vertices.
\end{itemize}

\end{lemma}

The total CPU time for the search for non-hamiltonian polyhedra was approximately $3$ years.
During the test for $19$ vertices alone $45,849,541,741,643$ polyhedra were tested for their
number of cuts and the existence of a hamiltonian cycle.
The search for non-traceable polyhedra took about $250$ days and
$12,229,809,370,343$ polyhedra
on $18$ vertices were tested for their number of cuts and the existence of hamiltonian paths.

\section*{Acknowledgements}

The computational resources (Stevin Supercomputer Infrastructure) used
to obtain the extensive computational results were provided by Ghent
University, the Hercules Foundation and the Flemish Government –
department EWI. We want to thank Brendan McKay for helpful discussions
on this topic.

\end{document}